\documentclass[12pt,showkeys]{amsart}
\usepackage{color}
\usepackage{graphicx}
\usepackage{subfigure}
\usepackage{amssymb}
\usepackage{amsmath}
\usepackage{multirow}
\usepackage{enumerate}
\usepackage{epstopdf}
\usepackage{cite,stmaryrd,txfonts}
\usepackage{tikz}

\input{undertilde}
\usetikzlibrary{snakes}

\usepackage{verbatim}

\usepackage{epic}

\usepackage{ulem}

\newcommand{\bs}{\boldsymbol}
\newcommand{\mrm}{\mathrm}

\def\cequiv{\raisebox{-1.5mm}{$\;\stackrel{\raisebox{-3.9mm}{=}}{{\sim}}\;$}}

\def\utau{\undertilde{\tau}}
\def\uphi{\undertilde{\varphi}}
\def\upsi{\undertilde{\psi}}
\def\ueta{\undertilde{\eta}}
\def\uzeta{\undertilde{\zeta}}
\def\ualpha{\undertilde{\alpha}}
\def\ubeta{\undertilde{\beta}}

\def\ugamma{\undertilde{\gamma}}

\def\uf{\undertilde{f}}

\def\ut{\undertilde{t}}

\def\uu{\undertilde{u}}
\def\uw{\undertilde{w}}

\def\ux{\undertilde{x}}

\def\rot{{\rm rot}}
\def\curl{{\rm curl}}
\def\dv{{\rm div}}

\def\uH{\undertilde{H}}

\def\vgm12{\bs{V}^{1+,2}_{\gamma,M}}

\newtheorem{theorem}{Theorem}
\newtheorem{remark}[theorem]{Remark}

\newtheorem{lemma}[theorem]{Lemma}

\newcounter{mnote}
\let\oldmarginpar\marginpar
\renewcommand\marginpar[1]{\-\oldmarginpar[\raggedleft\footnotesize #1]%
  {\raggedright\footnotesize #1}}

\begin{document}

\title{Stable mixed element schemes for plate models on multiply-connected domains}
\author{Shuo Zhang}
\address{LSEC, Institute of Computational Mathematics and Scientific/Engineering Computing, Academy of Mathematics and System Sciences, Chinese Academy of Sciences, Beijing 100190, People's Republic of China}
\email{szhang@lsec.cc.ac.cn}
\thanks{The author is supported partially by the National Natural Science Foundation of China with Grant No. 11471026 and National Centre for Mathematics and Interdisciplinary Sciences, Chinese Academy of Sciences.}




\begin{abstract} 
In this paper, we study the mixed element schemes of the Reissner-Mindlin plate model and the Kirchhoff plate model in multiply-connected domains. By a regular decomposition of $H_0(\rot,\Omega)$ and a Helmholtz decomposition of its dual, we develop mixed formulations of the models which are equivalent to the primal ones respectively and which are uniformly stable. A framework of designing uniformly stable finite element schemes is presented, and a specific example is given.
\end{abstract}

\maketitle


%
%
%
\section{Introduction}

In this paper, we study the Reissner-Mindlin model for moderately thick plates and the the Kirchhoff model for thin plates on multiply-connected domains. Among the many plate models in structural analysis, these two fall in the most frequently used ones. It is known that the stability of the Reissner-Mindlin model in its primal formulation is of a complicated representation, and utilising mixed formulations with auxiliary variables introduced is an important approach in the study of the model. This way, we will discuss the mixed element scheme of the Reissner-Mindlin model. The Kirchhoff model falls into the category of fourth order elliptic problem, and there have been many conforming and nonconforming finite elements for that. However, mixed element discretisation can bring in flexibility in implementation by finite element package and designing multilevel methods. Formally, the Kirchhoff model is the asymptotic limit of the Reissner-Mindlin model as the thickness tends to zero; we will also present mixed element schemes for the Kirchhoff model as the formal limit of that of the Reissner-Mindlin model.

There have been large literature on the mathematical analysis and numerical methods on Reissner-Mindlin model; we refer to \cite{Falk.R2008} for a brief review. The mathematical analysis and numerical solution of the model constructed on convex simply-connected polygons have been studied well. Since  \cite{Brezzi.F;Fortin.M1986,Brezzi.F;Fortin.M;Stenberg.R1991,Brezzi.F;Bathe.K;Fortin.M1989,Arnold.D;Falk.R1989}, procedures for developing stable and convergent finite element methods have been firmly established. The fast solution of the generated finite element system is discussed in \cite{Arnold.D;Falk.R;Winther.R1997}. For these achievements, mixed formulations played important roles \cite{Brezzi.F;Fortin.M1986,Brezzi.F;Fortin.M;Stenberg.R1991,Arnold.D;Falk.R1989}. Some other mixed formulation can be found in, e.g., \cite{Hughes.T;Franca.L1988,Lovadina.C1996,Taylor.R;Auricchio.F1993,Amara.M;Capatina-Papaghiuc.D;Chatti.A2002,Behrens.E;Guzman.J2011new}. 

In contrast, when the domain is multiply-connected (thus non convex), as known by the author, the model has not been discussed though it is practically applicable. An important difference that lies between simply- and multiply- connected domains is that a space of harmonic functions is contained in the Helmholtz decomposition of, e.g., $H_0(\rot,\Omega)$, and procedures developed in \cite{Brezzi.F;Fortin.M1986,Brezzi.F;Fortin.M;Stenberg.R1991,Brezzi.F;Bathe.K;Fortin.M1989,Arnold.D;Falk.R1989} can not trivially be repeated whereas the Helmholtz decomposition plays a crucial role. The influence of the existence of harmonic functions has been discussed in the context of Maxwell equation, for which we refer to, e.g., \cite{Brenner.S;Cui.J;Nan.Z;Sung.L2012,Pasciak.J;Zhao.J2002} for related discussion, while its influence on Reissner-Mindlin plate has not been discussed. Some investigation on the problem is carried on in the present paper. It is verified that $H_0(\rot,\Omega)=\nabla H^1_0(\Omega)+(H^1_0(\Omega))^2$ still holds on multiply-connected domains. Based on this observation, also to deal with the obstacle of extra harmonic functions, we suggest a new mixed formulation for the Reissner-Mindlin model, and prove its uniform stability. The new mixed formulation is different from the ones aforementioned even when restricted to simply-connected domains. Further, a framework of constructing finite element schemes based on the mixed formulation is presented. The error estimation in energy norm follows with respect to the regularity of the system. A specific example is given in the framework.

The remaining of the paper is organised as follows. In Section \ref{sec:sobspa}, we study the Sobolev space $H_0(\rot,\Omega)$, and construct a regular decomposition of $H_0(\rot,\Omega)$ and a Helmholtz decomposition of $H_0(\rot,\Omega)'$. In Section \ref{sec:mixform}, a uniformly stable mixed formulation of the Reissner-Mindlin plate model and a stable mixed formulation of the Kirchhoff plate model are constructed based on the structural properties presented in Section \ref{sec:mixform}. Then in Section \ref{sec:fem}, finite element discretizations of the models are discussed. Several general conditions are presented for selecting finite element spaces to form discretisations for thick and thin plate models, and a specific example that satisfies the conditions are given. Coherently, a primal scheme which is coincident with the Dur\'an-Liberman scheme \cite{Duran.R;Liberman.E1992} designed on simply-connected domains is derived by the new approach for multiply-connected domains. And finally, some concluding remarks are given in Section \ref{sec:conc}.

\section{Structure of Sobolev spaces on multiply-connected polygon}
\label{sec:sobspa}

\subsection{Preliminaries}

Through this paper, we use $\Omega$ for a multiply-connected polygonal domain. Specifically, let $\Omega_0\subset\mathbb{R}^2$ be a simply-connected polygon with boundary $\Gamma_0$, and $\Omega_j\subset\Omega_0$ be simply-connected polygons with boundary $\Gamma_j$, $j=1,\dots,J$, such that ${\rm dist}(\Gamma_i,\Gamma_j)>0$ for any $0\leqslant i\neq j\leqslant J$, and define $\Omega=\Omega_0\setminus(\cup_{j=1}^J\overline{\Omega}_j)$. Evidently, $\Omega$ is a bounded connected domain in $\mathbb{R}^2$ with Lipschitz boundary. Denote by $\Gamma$ the boundary of $\Omega$; then $\Gamma=\cup_{j=0}^J\Gamma_j$. In this paper, we consider the Reissner-Mindlin and the Kirchhoff plate model on $\Omega$.

We use $\nabla$, $\curl$, $\dv$ and $\rot$ for the gradient operator, curl operator, divergence operator and rot operator, respectively; $\curl$ and $\rot$ are perpendicular to $\nabla$ and $\dv$, respectively. As usual, we use $H^2(\Omega)$, $H^2_0(\Omega)$, $H^1(\Omega)$, $H^1_0(\Omega)$, $H(\rot,\Omega)$, $H_0(\rot,\Omega)$ and $L^2(\Omega)$ for certain Sobolev spaces, and specifically, denote  $\displaystyle L^2_0(\Omega):=\{w\in L^2(\Omega):\int_\Omega w dx=0\}$, $\undertilde{H}{}^1_0(\Omega):=(H^1_0(\Omega))^2$, $\mathring{H}_0(\rot,\Omega):=\{\utau\in H_0(\rot,\Omega):\rot\utau=0\}$, and $\mathring{\uH}{}^1_0(\Omega):=\{\upsi\in \uH{}^1_0(\Omega):\rot\upsi=0\}$. Denote by $(\mathring{\uH}{}^1_0(\Omega))^\perp$ the orthogonal complement of $\mathring{\uH}{}^1_0(\Omega)$ in $\uH{}^1_0(\Omega)$ with respect to the inner product $(\nabla\cdot,\nabla\cdot)$, and by $(\mathring{H}_0(\rot,\Omega))^\perp$ the orthogonal complement of $\mathring{H}_0(\rot,\Omega)$ in $H_0(\rot,\Omega)$ with respect to the inner product $(\cdot,\cdot)$ and simultaneously the inner product of $H_0(\rot,\Omega)$. With respect to the multiply-connectivity, define
\begin{itemize}
\item $H^1_C(\Omega):=\{w\in H^1(\Omega):w|_{\Gamma_0}=0,\ w|_{\Gamma_j}=\mbox{constant},\ 1\leqslant j\leqslant J\}$;
\item $\mathcal{H}^1_C(\Omega):=\{w\in H^1_C(\Omega):(\nabla w,\nabla s)=0,\ \forall\,s\in H^1_0(\Omega)\}$;
\item $H^2_C(\Omega):=\{w\in H^2(\Omega):w|_{\Gamma_0}=0,\ w|_{\Gamma_j}=\mbox{constant},\ 1\leqslant j\leqslant J,\ \frac{\partial w}{\partial\mathbf{n}}=0\ \mbox{on}\ \Gamma\}$;
\item $\mathcal{H}^2_C(\Omega):=\{w\in H^2_C(\Omega):(\nabla^2 w,\nabla^2 s)=0,\ \forall\,s\in H^2_0(\Omega)\}$.
\end{itemize}
By elliptic regularity theory\cite{Nazarov.S;Plamenevsky.B1994,Grisvard.P1985,Dauge.M2006,Brenner.S;Cui.J;Nan.Z;Sung.L2012}, $\psi\in H^{3/2+\delta_0}(\Omega)$ for some $\delta_0>0$ if $\psi\in \mathcal{H}^1_C(\Omega)$. We use $``\undertilde{~}"$ for vector valued quantities in the present paper. We use $(\cdot,\cdot)$ for $L^2$ inner product and $\langle\cdot,\cdot\rangle$ for the duality between a space and its dual. Without ambiguity, we use the same notation $\langle\cdot,\cdot\rangle$ for different dualities, and it can occasionally be treated as $L^2$ inner product for certain functions. And finally, $\lesssim$, $\gtrsim$, and $\cequiv$ respectively denote $\leqslant$, $\geqslant$, and $=$ up to a constant. The hidden constants depend on the domain, and, when triangulation is involved, they also depend on the shape-regularity of the triangulation, but they do not depend on $h$ or any other mesh parameter.

Recall that $\rot$ is the rotation of $\dv$. By virtue of Corollaries 3.1 and 3.2 and then Corollary 2.4 of \cite{Girault.V;Raviart.P1986}, we have the lemma below.
\begin{lemma}\label{lem:isorot}
\begin{enumerate}
\item
$\mathring{H}_0(\rot,\Omega)=\nabla H^1_C(\Omega)$; $\mathring{\uH}{}^1_0(\Omega)=\nabla H^2_C(\Omega)$.
\item $\rot$ is an isomorphism from $(\nabla H^2_C(\Omega))^\perp$ onto $L^2_0(\Omega)$.
\end{enumerate}
\end{lemma}

The Friedrichs' inequality below follows from Lemma \ref{lem:isorot}.
\begin{lemma}\label{lem:fine}
There exists a constant $C$, such that $\|\utau\|_{\rot,\Omega}\leqslant C\|\rot\utau\|_{0,\Omega}$ for $\utau\in (\nabla H^1_C(\Omega))^\perp$.
\end{lemma}

\subsection{Regular decomposition of $H_0(\rot,\Omega)$}
First of all, the spaces $\mathcal{H}^1_C(\Omega)$ and $\mathcal{H}^2_C(\Omega)$ have the same dimension $J$. Any two norms on each of them are equivalent.
\begin{lemma}\label{lem:findim}
Let $\|\cdot\|_A$ and $\|\cdot\|_B$ be two norms defined on $\mathcal{H}^1_C$ and $\mathcal{H}^2_C$, respectively. There exist two constants $C_s$ and $C_b$, such that, if $w_i\in\mathcal{H}^i_C$ and $w_1|_\Gamma=w_2|_\Gamma$, then $C_s\|w_1\|_{A,\Omega}\leqslant \|w_2\|_{B,\Omega}\leqslant C_b\|w_1\|_{A,\Omega}$.
\end{lemma}

\begin{proof}
For $i=1,2$, define $\Upsilon_i$ from $\mathbb{R}^J$ to $\mathcal{H}^i_C$ by
$$
\Upsilon_i\undertilde{\upsilon}\in\mathcal{H}^i_C,\ \ \Upsilon_i\undertilde{\upsilon}|_{\Gamma_j}=(\undertilde{\upsilon})(j),\ \ j=1,\dots,J,\ \ \undertilde{\upsilon}\in\mathbb{R}^J.
$$
Moreover, $\Upsilon_i\undertilde{\upsilon}=0$ iff $\undertilde{\upsilon}=\undertilde{0}$, $i=1,2$. Therefore, two norms on $\mathbb{R}^J$ can be defined by
$$
\|\undertilde{\upsilon}\|_{*}:=\|\Upsilon_1\undertilde{\upsilon}\|_{A,\Omega}, \ \ \mbox{and}\ \ \|\undertilde{\upsilon}\|_{**}:=\|\Upsilon_2\undertilde{\upsilon}\|_{B,\Omega}.
$$
As $\mathbb{R}^J$ is of finite dimensional, $\|\undertilde{\upsilon}\|_*$ and $\|\undertilde{\upsilon}\|_{**}$ are equivalent. This completes the proof. 
\end{proof}

\begin{theorem}\label{thm:sdhrot}(Stable regular decomposition of $H_0(\rot,\Omega)$) Given $\utau\in H_0(\rot,\Omega)$, there exist $w_\tau\in H^1_0(\Omega)$ and $\uphi{}_\tau\in\uH{}^1_0(\Omega)$, such that $\|w_\tau\|_{1,\Omega}+\|\uphi{}_\tau\|_{1,\Omega}\leqslant C\|\utau\|_{\rot,\Omega}$, and $\nabla w_\tau+\uphi{}_\tau=\utau$.
\end{theorem}

\begin{proof}
Given $\utau\in H_0(\rot,\Omega)$, by Lemma \ref{lem:isorot}, there exists a unique $\uphi\in(\nabla H^2_C(\Omega))^\perp$, such that $\rot\uphi=\rot\utau$ and $\|\uphi\|_{1,\Omega}\leqslant C\|\rot\utau\|_{0,\Omega}$. Then there exists a $w\in H^1_C(\Omega)$, such that $\nabla w=\utau-\uphi$. Evidently, $\|w\|_{1,\Omega}\leqslant C\|\utau\|_{\rot,\Omega}$. Decompose $w=w_1+w_2$ with $w_1\in H^1_0(\Omega)$ and $w_2\in\mathcal{H}^1_C(\Omega)$, then $\|\nabla w_1\|_{0,\Omega}+\|\nabla w_2\|_{0,\Omega}\leqslant 2\|\nabla w\|_{0,\Omega}.$ Further, choose $w_3\in\mathcal{H}^2_C(\Omega)$ such that $w_3|_\Gamma=w_2|_\Gamma$, then $w_2=w_3+w_4$ with $w_4\in H^1_0(\Omega)$. Direct calculation leads to that $(\nabla w_3,\nabla w_3)=(\nabla w_2,\nabla w_2)+(\nabla w_4,\nabla w_4)$. Thus $\|\nabla w_4\|_{0,\Omega}\leqslant\|\nabla w_3\|_{0,\Omega}\leqslant C\|\nabla w_2\|_{0,\Omega}$ by Lemma \ref{lem:findim}. Now we arrive at the decomposition
$$
\utau=\nabla w+\uphi=\nabla (w_1+w_4)+(\nabla w_3+\uphi),
$$
where $w_1+w_4\in H^1_0(\Omega)$, and $\|w_1+w_4\|_{1,\Omega}\leqslant \|w_1\|_{1,\Omega}+\|w_4\|_{1,\Omega}\leqslant \|w_1\|_{1,\Omega}+C\|w_2\|_{1,\Omega}\leqslant C\|\utau\|_{\rot,\Omega}$
and
$\nabla w_3+\uphi\in\uH{}^1_0(\Omega)$, and $\|\nabla w_3+\uphi\|_{1,\Omega}\leqslant \|\nabla w_3\|_{1,\Omega}+\|\uphi\|_{1,\Omega}\leqslant C\|\nabla w_3\|_{0,\Omega}+\|\uphi\|_{1,\Omega}\leqslant C\|\utau\|_{\rot,\Omega}$, where we have used Lemma \ref{lem:findim} again. Taking $w_\tau:=w_1+w_4$ and $\uphi{}_\tau:=\nabla w_3+\uphi$ completes the proof. 
\end{proof}
\begin{remark}
Theorem \ref{thm:sdhrot} states actually 
\begin{equation}\label{eq:sdhrot}
H_0(\rot,\Omega)=\nabla H^1_0(\Omega)+\uH{}^1_0(\Omega).
\end{equation}
As evidently $\displaystyle\|\utau\|_{\rot,\Omega}\leqslant C\inf_{\substack{w\in H^1_0(\Omega),\uphi\in\uH{}^1_0(\Omega), \utau=\nabla w+\uphi}}\|w\|_{1,\Omega}+\|\uphi\|_{1,\Omega}$, the equivalence \eqref{eq:sdhrot} can be proved by the open mapping theorem. Here we present a constructive proof instead. By Lemma \ref{lem:isorot}, another stable decomposition $H_0(\rot,\Omega)=\nabla H^1_C(\Omega)+\uH{}^1_0(\Omega)$ can be derived directly. Similar decomposition can be found discussed in, e.g., \cite{Pasciak.J;Zhao.J2002}.
\end{remark}

\subsection{Helmholtz decomposition of $H_0(\rot,\Omega)'$}

Define 
$$
H^{-1}(\dv,\Omega):=\{\ueta\in (\uH{}^1_0(\Omega))':\dv\ueta\in (H^1_0(\Omega))'\}.
$$
By Theorem \ref{thm:sdhrot}, 
$$
H^{-1}(\dv,\Omega)=(H_0(\rot,\Omega))'.
$$

\begin{theorem}\label{thm:hdh-1div}(Helmholtz decomposition of $H_0(\rot,\Omega)'$) The Helmholtz decomposition holds
\begin{equation}
H^{-1}(\dv,\Omega)=\nabla H^1_C(\Omega)+\curl L^2_0(\Omega).
\end{equation}
Namely, given $\ueta\in H^{-1}(\dv,\Omega)$, there exists uniquely a $w_\eta\in H^1_C(\Omega)$ and $p\in L^2_0(\Omega)$, such that 
$\ueta=\nabla w_\eta+\curl p,$ and moreover, $\|\ueta\|_{H^{-1}(\dv,\Omega)}\cequiv\|w_\eta\|_{1,\Omega}+\|p\|_{0,\Omega}.$
\end{theorem}
\begin{proof}
Given $\ueta\in H^{-1}(\dv,\Omega)$, there exists a $w_\eta\in H^1_C(\Omega)$, such that 
$$
(\nabla w_\eta,\nabla v)=\langle\ueta,\nabla v\rangle,\ \ \forall\,v\in H^1_C(\Omega).
$$
Thus $\langle\ueta-\nabla w_\eta,\upsi\rangle=0$ if $\upsi\in\mathring{H}_0(\rot,\Omega)$, and therefore there exists a $p\in L^2_0(\Omega)$, such that 
$$
(p,\rot\upsi)=\langle \ueta-\nabla w_\eta,\upsi\rangle,\ \ \forall\,\upsi\in H_0(\rot,\Omega).
$$
Namely $\ueta-\nabla w_\eta=\curl p$. The norm equivalence follows immediately by Lemma \ref{lem:fine}.
\end{proof}

\begin{remark}
An orthogonal decomposition reads $\nabla H^1_C(\Omega)=\nabla H^1_0(\Omega)\oplus\nabla\mathcal{H}^1_C(\Omega)$. If $\utau\in\nabla\mathcal{H}^1_C(\Omega)$, $\dv\utau=\rot\utau=0$. Namely $\nabla\mathcal{H}^1_C(\Omega)$ is the harmonic component of the Helmhotlz decomposition.
\end{remark}

\section{Mixed formulations of the thick and thin plate models}
\label{sec:mixform}

\subsection{Model problems}

In this paper, we consider the Reissner-Mindlin plate model of the form 
\begin{equation}\label{eq:RMForm1}
\left\{
\begin{array}{rr}
-{\rm div}\mathcal{C}\mathcal{E}\uphi+\lambda t^{-2}(\uphi-\nabla \omega)&=0, \\ 
\lambda t^{-2}(-\Delta \omega+{\rm div}\uphi) & =g,
\end{array}
\right.
\end{equation}
on $\Omega$ together with conditions $\omega=0$ and $\uphi=\undertilde{0}$ for the hard clamped plate. Mechanically, $g$ is the scaled transverse loading function, $t$ is the plate thickness, $\mathcal{E}\phi$ is the symmetric part of the gradient of $\phi$, and the scalar constant $\lambda$ and constant tensor $\mathcal{C}$ depend on the material properties of the body. Usually, $\lambda=Ek/2(1+\nu)$ with $E$ Young's modulus, $\nu$ the Poisson ratio, and $k$ the shear correction factor. For all $2\times 2$ symmetric matrices $\tau$, $\mathcal{C}\tau$ is defined by 
$$
\mathcal{C}\tau = \frac{E}{12(1-\nu^2)}[(1-\nu)\tau+\nu{\rm tr}(\tau) Id].
$$

Mathematically, we consider the variational problem: given $\uf\in \uH{}^{-1}(\Omega)$ and $g\in H^{-1}(\Omega)$, to find $(\uphi^t,\omega^t)\in\uH{}^1_0(\Omega)\times H^1_0(\Omega)$, such that
\begin{equation}\label{eq:vfrmo}
(\mathcal{CE}(\uphi^t),\mathcal{E}(\upsi))+\lambda t^{-2}(\uphi^t-\nabla\,\omega^t,\upsi-\nabla\, \mu)=\langle\uf,\upsi\rangle+\langle g,\mu\rangle\ \ \forall\,(\upsi,\mu)\times \uH{}^1_0(\Omega)\times H^1_0(\Omega).
\end{equation}
In the sequel, for simplicity, we just take $\lambda=1$.

At the limit as $t$ tends to zero, we consider the Kirchhoff plate model: find $\omega^0\in H^2_0(\Omega)$, such that 
\begin{equation}
(\mathcal{CE}(\nabla\omega^0),\mathcal{E}(\mu))=\langle\uf,\nabla\mu\rangle+\langle g,\mu\rangle,\ \ \forall\,\mu\in H^2_0(\Omega).
\end{equation}

\subsection{A mixed formulation of the Reissner-Mindlin plate}

For $(\uphi^t,\omega^t)\in \uH{}^1_0(\Omega)\times H^1_0(\Omega)$ and $t>0$, introduce the shear force
\begin{equation}
\uzeta^t=t^{-2}(\nabla \omega^t-\uphi^t),
\end{equation}
then $\uzeta^t\in H_0(\rot,\Omega)$. Now denote, for $t\geqslant 0$,
\begin{equation}
Z_t:=\{(\upsi,\mu,\ueta)\in\uH{}^1_0(\Omega)\times H^1_0(\Omega)\times H_0(\rot,\Omega):t^2\ueta=(\nabla\mu-\upsi)\},
\end{equation} 
and the problem \eqref{eq:vfrmo} can be equivalently rewritten as: finding $(\uphi^t,\omega^t,\uzeta^t)\in Z_t$, such that 
\begin{equation}\label{eq:rmt}
(\mathcal{CE}(\uphi^t),\mathcal{E}(\upsi))+ t^2(\uzeta^t,\ueta)=\langle \uf,\upsi \rangle+\langle g,\mu\rangle,\quad\forall\,(\upsi,\mu,\ueta)\in Z_t.
\end{equation}
Note that, by Theorem \ref{thm:hdh-1div},  $(\upsi,\mu,\ueta)\in Z_t$ is equivalent to 
$$
(t^2\ueta-\nabla\mu+\upsi,\nabla z)+(\rot\,(t^2\ueta-\nabla \mu+\upsi),q)=0,\ \ \forall\,z\in H^1_C(\Omega),\ q\in L^2_0(\Omega). 
$$

Now, we introduce the Lagrangian multiplier $y^t\in H^1_C(\Omega)$ and $p^t\in L^2_0(\Omega)$, and rewrite \eqref{eq:rmt} to: finding $(\uphi^t,\omega^t,\uzeta^t,y^t,p^t)\in X^t:=\uH{}^1_0(\Omega)\times H^1_0(\Omega)\times (t\undertilde{L}^2(\Omega)\cap t^2H_0(\rot,\Omega))\times H^1_C(\Omega)\times L^2_0(\Omega)$, such that, for $(\upsi,\mu,\ueta,z,q)\in X^t$,
\begin{multline}\label{eq:vpaug1}
(\mathcal{CE}(\uphi^t),\mathcal{E}(\upsi))+ t^2(\uzeta^t,\ueta)+ (t^2\ueta-\nabla\mu+\upsi,\nabla y^t)+(\rot\,(t^2\ueta-\nabla \mu+\upsi),p^t) 
\\
+(t^2\uzeta^t-\nabla\omega^t+\uphi^t,\nabla z) +(\rot\,(t^2\uzeta^t-\nabla \omega^t+\uphi^t),q)=\langle \uf,\upsi \rangle+\langle g,\mu\rangle.
\end{multline}
Note that $t\undertilde{L}^2(\Omega)\cap t^2H_0(\rot,\Omega)$ coincides to $H_0(\rot,\Omega)$ for $t>0$, but equipped with a different norm $t\|\cdot\|_{0,\Omega}+t^2\|\rot\cdot\|_{0,\Omega}$. For $(\upsi,\mu,\ueta,z,q)\in X^t$, 
$$
\|(\upsi,\mu,\ueta,z,q)\|_{X^t}=\|\upsi^t\|_{1,\Omega}+\|\mu\|_{1,\Omega}+t\|\ueta\|_{0,\Omega}+t^2\|\rot\ueta\|_{0,\Omega}+\|z\|_{1,\Omega}+\|q\|_{0,\Omega}.
$$

\subsubsection{Well-posedness of the system}

\begin{theorem}\label{thm:wprmmix}
Given $\uf\in \uH{}^{-1}(\Omega)$ and $g\in H^{-1}(\Omega)$, there exists a unique $(\uphi^t,\omega^t,\uzeta^t,y^t,p^t)\in X^t$ that satisfies \eqref{eq:vpaug1}, and
\begin{equation}\label{eq:stabyRM}
\|(\uphi^t,\omega^t,\uzeta^t,y^t,p^t)\|_{X_t}\cequiv \|\uf\|_{-1,\Omega}+\|g\|_{-1,\Omega}.
\end{equation}
Moroever, $(\uphi^t,\omega^t)$ solves \eqref{eq:vfrmo}.
\end{theorem}
\begin{proof}
We study the well-posed-ness of the problem below: find $(\hat\uphi^t,\hat\omega^t,\hat\uzeta^t,\hat y^t,\hat p^t)\in X^t$, such that, for $(\upsi,\mu,\ueta,z,q)\in X^t$,
\begin{equation}
\left\{
\begin{array}{cccccll}
(\mathcal{CE}(\hat\uphi^t),\mathcal{E}(\upsi)) && &+(\upsi,\nabla\hat y^t) &+(\rot\,\upsi,\hat p^t) & = \langle\uf{}_\psi,\upsi\rangle
\\
& t^2(\hat\uzeta^t,\ueta)&&+ t^2(\eta,\nabla\hat y^t) &+t^2(\rot\,\ueta,\hat p^t)  & =\langle\uf{}_\eta,\ueta\rangle
\\
&& &-(\nabla\mu,\nabla\hat y^t)&& =\langle f_\mu,\mu\rangle
\\
(\hat\uphi^t,\nabla z)&+t^2(\hat\zeta^t,\nabla z) &-(\nabla\hat\omega^t,\nabla z) && & = \langle f_z,z\rangle
\\
(\rot\,\hat\uphi^t,q) & +t^2(\rot\,\hat\uzeta^t,q) &&& & =\langle f_q,q\rangle.
\end{array}
\right.
\end{equation}
We are going to show that 
\begin{equation}
\|(\hat\uphi^t,\hat\omega^t,\hat\uzeta^t,\hat y^t,\hat p^t)\|_{X_t}\cequiv \|(\uf{}_\psi,\uf{}_\eta,f_\mu,f_z,f_q)\|_{X_t'},
\end{equation}
for which we only have to verify Brezzi's conditions, and \eqref{eq:stabyRM} follows. Define
\begin{equation}
a((\uphi,\uzeta,\omega),(\upsi,\ueta,\mu)):=(\mathcal{CE}(\uphi),\mathcal{E}(\upsi))+(\uzeta,\ueta),
\end{equation}
and
\begin{equation}
b_t((\uphi,\uzeta,\omega),(z,q)):=(\uphi,\nabla z)+t^2(\zeta,\nabla z) -(\nabla\omega,\nabla z) +(\rot\,\uphi,q)  +t^2(\rot\,\uzeta,q).
\end{equation}
Then $Z_t=\{(\upsi,\mu,\ueta)\in\uH{}^1_0(\Omega)\times H^1_0(\Omega)\times H_0(\rot,\Omega):b_t((\upsi,\ueta,\mu),(z,q))=0,\ \forall\,(z,q)\in H^1_C(\Omega)\times L^2_0(\Omega)\}$. It is evident that 
\begin{multline*}
a((\uphi,\uzeta,\omega),(\upsi,\ueta,\mu))
\\
\leqslant (\|\uphi\|_{1,\Omega}+t\|\uzeta\|_{0,\Omega}+t^2\|\rot\uzeta\|_{0,\Omega}+\|\omega\|_{1,\Omega})(\|\upsi\|_{1,\Omega}+t\|\ueta\|_{0,\Omega}+t^2\|\rot\ueta\|_{0,\Omega}+\|\mu\|_{1,\Omega}),
\end{multline*}
$$
\mbox{and}\ \ 
b_t((\uphi,\uzeta,\omega),(z,q))\leqslant (\|\uphi\|_{1,\Omega}+t\|\uzeta\|_{0,\Omega}+t^2\|\rot\uzeta\|_{0,\Omega}+\|\omega\|_{1,\Omega})(\|z\|_{1,\Omega}+\|q\|_{0,\Omega}).
$$
Meanwhile, 
$$
a((\uphi,\uzeta,\omega),(\uphi,\uzeta,\omega))\geqslant C (\|\uphi\|_{1,\Omega}+t\|\uzeta\|_{0,\Omega}+t^2\|\rot\uzeta\|_{0,\Omega}+\|\omega\|_{1,\Omega})^2, \mbox{for}\ (\uphi,\uzeta,\omega)\in Z_t.
$$
It remains for us to show the inf-sup condition, which reads
\begin{equation}\label{eq:infsuprm}
\sup_{(\upsi,\mu,\ueta)\in\uH{}^1_0(\Omega)\times H^1_0(\Omega)\times H_0(\rot,\Omega)\setminus\{\bf 0\}}\frac{b_t((\uphi,\uzeta,\omega),(z,q))}{\|\uphi\|_{1,\Omega}+t\|\uzeta\|_{0,\Omega}+t^2\|\rot\uzeta\|_{0,\Omega}+\|\omega\|_{1,\Omega}}\geqslant C(\|z\|_{1,\Omega}+\|q\|_{0,\Omega})
\end{equation}
for any $(z,q)\in H^1_C(\Omega)\times L^2_0(\Omega)\setminus\{\mathbf{0}\}$. Given $(z,q)\in H^1_C(\Omega)\times L^2_0(\Omega)$, decompose $z=z_1+z_2$ with $z_1\in\mathcal{H}^1_C$ and $z_2\in H^1_0(\Omega)$ and choose 
\begin{itemize}
\item $\uphi{}_1\in (\nabla H^2_C(\Omega))^\perp$, such that $\rot\uphi{}_1=q$;
\item $\uphi{}_2=\nabla\Phi$, with $\Phi\in\mathcal{H}^2_C$, such that $(\uphi{}_2,\nabla s)=(\nabla z_1,\nabla s)-(\uphi{}_1,\nabla s)$ for any $s\in\mathcal{H}^1_C$;
\item $\uphi=\uphi{}_1+\uphi{}_2$;
\item $\omega\in H^1_0(\Omega)$ such that $(\nabla\omega,\nabla s)=(\uphi-\nabla z_2,\nabla s)$ for any $s\in H^1_0(\Omega)$;
\item $\uzeta=\undertilde{0}$.
\end{itemize}
For any $\Psi\in \mathcal{H}^2_C$, we can choose $\psi\in \mathcal{H}^1_C$, such that $\psi=\Psi$ on $\Gamma$, and then $(\nabla \Psi,\nabla s)=(\nabla\psi,\nabla s)$ for any $s\in \mathcal{H}^1_C$; this guarantees the existence of $\uphi{}_2$. Then 
\begin{equation}
b((\uphi,\uzeta,\omega),(z,q))=\|\nabla z\|_{0,\Omega}^2+\|q\|_{0,\Omega}^2.
\end{equation}
Meanwhile, $\|\uphi{}_1\|_{1,\Omega}\leqslant C\|q\|_{0,\Omega}$, $\|\uphi{}_2\|_{1,\Omega}\leqslant C\|\uphi{}_2\|_{0,\Omega}\leqslant C(\|z\|_{1,\Omega}+\|q\|_{0,\Omega})$, and $\|\omega\|_{1,\Omega}\leqslant C(\|z\|_{1,\Omega}+\|q\|_{0,\Omega})$. This confirms the inf-sup condition \eqref{eq:infsuprm} and completes the proof. 
\end{proof}

\subsubsection{Comparison with Brezzi-Fortin-Stenberg's mixed formulation}
Following the line in \cite{Brezzi.F;Fortin.M;Stenberg.R1991}, we can compose a mixed formulation of the model problem \eqref{eq:vfrmo}, which reads: given $\uf\in \uH{}^{-1}(\Omega)$ and $g\in H^{-1}(\Omega)$, to find $(\uphi{}_{\rm BFS}^t,\omega{}_{\rm BFS}^t,\ualpha{}_{\rm BFS}^t,y{}_{\rm BFS}^t,p{}_{\rm BFS}^t)\in X^t$, such that, for $(\upsi,\mu,\ubeta,z,q)\in X^t$, 
\begin{equation}\label{eq:BFSformulation}
\left\{
\begin{array}{cccccll}
(\mathcal{CE}(\uphi{}_{\rm BFS}^t),\mathcal{E}(\upsi))&&-(\nabla y{}_{\rm BFS}^t,\upsi)&-(p{}_{\rm BFS}^t,\rot\upsi)&&=&\langle\uf,\upsi\rangle
\\
&t^2(\ualpha{}_{\rm BFS}^t,\ubeta)&&-t^2(p{}_{\rm BFS}^t,\rot\ubeta)&&=&0
\\
-(\uphi{}_{\rm BFS}^t,\nabla z)&&-t^2(\nabla y{}_{\rm BFS}^t,\nabla z)&&(\nabla\omega{}_{\rm BFS}^t,\nabla z)&=&0 
\\
-(\rot\uphi{}_{\rm BFS}^t,q)&-t^2(\rot\ualpha{}_{\rm BFS}^t,q)&&&&=&0
\\
&&(\nabla y{}_{\rm BFS}^t,\nabla v)&&&=&\langle g,v\rangle.
\end{array}
\right.
\end{equation}
This system is the same as (2.16) through (2.18) of \cite{Brezzi.F;Fortin.M;Stenberg.R1991}, up to an $H^1_0(\Omega)$ replaced by $H^1_C(\Omega)$. It can be observed that:
\begin{enumerate}
\item systems \eqref{eq:vpaug1} and \eqref{eq:BFSformulation} are the same when $t=0$;
\item once $y^t$ ($y^t_{\rm BFS}$, respectively) is known for System \eqref{eq:vpaug1}(System \eqref{eq:BFSformulation}, respectively), both the two systems can be decoupled, and the decoupled subsystems are the same;
\item once $y^t$ ($y^t_{\rm BFS}$, respectively) can be decoupled from the entire system, the regularity analysis of the two systems are the same;
\item if $(\uphi^t,\omega^t,\uzeta^t,y^t,p^t)\in X^t$ and $(\uphi{}_{\rm BFS}^t,\omega{}_{\rm BFS}^t,\ualpha{}_{\rm BFS}^t,y{}_{\rm BFS}^t,p{}_{\rm BFS}^t)\in X^t$ are the solutions of \eqref{eq:vpaug1} and \eqref{eq:BFSformulation}, respectively, then
\begin{equation*}
(\uphi^t,\omega^t)=(\uphi{}^t_{\rm BFS},\omega^t_{\rm BFS}),\ \  (y^t,p^t)=-(y^t_{\rm BFS},p^t_{\rm BFS}),\ \  \ualpha{}^t_{\rm BFS}=\nabla y^t-\uzeta^t.
\end{equation*}
\end{enumerate}
We remark that, as multiply-connected domains are under consideration, $y^t$ ($y^t_{\rm BFS}$, as well) can not be solved out simply, and neither of the systems can be decomposed.


%
%
\subsection{A mixed formulation of the Kirchhoff plate.}
Note that the space $Z_t$ makes sense for $t=0$. The Kirchhoff plate problem can be rewritten as: finding $(\uphi^0,\omega^0)\in Z_0$, such that 
\begin{equation}\label{eq:kp}
(\mathcal{E}(\uphi^0),\mathcal{CE}(\upsi))=\langle\uf,\upsi\rangle+\langle g,\mu\rangle,\ \ \forall\,(\upsi,\mu)\in Z_0.
\end{equation}
Again we can introduce Lagrangian multiplier, and have an expanded system: find $(\uphi^0,\omega^0,y^0,p^0)\in Y:=\uH{}^1_0(\Omega)\times H^1_0(\Omega)\times H^1_C(\Omega)\times L^2_0(\Omega)$, such that, for $(\upsi,\mu,z,q)\in Y$,
\begin{equation}\label{eq:mixkp}
\left\{
\begin{array}{cccccll}
(\mathcal{CE}(\uphi^0),\mathcal{E}(\upsi)) && &+(\upsi,\nabla y^0) &+(\rot\,\upsi,p^0) & = \langle\uf,\upsi\rangle,
\\
&& &-(\nabla\mu,\nabla y^0)&& =\langle g,\mu\rangle,
\\
(\uphi^0,\nabla z)&&-(\nabla\omega^0,\nabla z) && & = 0,
\\
(\rot\,\uphi^0,q) &&&& & =0.
\end{array}
\right.
\end{equation}
Similar to Theorem \ref{thm:wprmmix}, the theorem below surveys the well-posedness of \eqref{eq:mixkp}.
\begin{theorem}
Given $\uf\in\uH{}^{-1}(\Omega)$ and $g\in H^{-1}(\Omega)$, the problem \eqref{eq:mixkp} admits a unique solution $(\uphi^0,\omega^0,y^0,p^0)\in Y$, and $\|(\uphi^0,\omega^0,y^0,p^0)\|_Y\cequiv \|\uf\|_{-1,\Omega}+\|g\|_{-1,\Omega}$. Moreover, $\uphi^0=\nabla \omega^0$, and $(\uphi^0,\omega^0)$ solves \eqref{eq:kp}.
\end{theorem}

\section{Finite element discretisation of the plate models}
\label{sec:fem}

\subsection{A general construction of mixed finite element discretization}
Given a subdivision of $\Omega$, let $\uH{}^1_{h0}\subset \uH{}^1_0$,  $H^1_h\subset H^1(\Omega)$,  $H_{h0}(\rot)\subset H_0(\rot,\Omega)$, and $L^2_{h0}\subset L^2_0(\Omega)$ be respective finite element spaces. Set $H^1_{hC}:=H^1_h\cap H^1_C(\Omega)$ and $H^1_{h0}:=H^1_h\cap H^1_0(\Omega)$. 

For the well-posed-ness of the finite element schemes, we introduce these assumptions.
\begin{description}
\item[{\textbf A1}] There exists a Fortin operator $\Pi_h^{\mathsf F}:\uH{}^1_0(\Omega)\to \uH{}^1_{h0}$, such that 
\begin{equation}
(\rot\Pi_h^{\mathsf F}\uphi,q_h)=(\rot\uphi,q_h),\ \forall\,q_h\in L^2_{h0},\ \ \ |\Pi_h^{\mathsf F}\uphi-\uphi|_{k,\Omega}\leqslant Ch^{1-k}|\uphi|_{1,\Omega},\ k=1,2.
\end{equation}
\item[{\textbf A2}] $\displaystyle\inf_{v_h\in H^1_h}\|w-v_h\|_{1,\Omega}\leqslant Ch^{1/2+\delta_0}\|w\|_{3/2+\delta_0,\Omega}$ for $w\in H^{3/2+\delta_0}(\Omega)$.
\item[{\textbf A3}] $\rot H_{h0}(\rot)=L^2_{h0}$.
\item[{\textbf A4}] $\nabla H^1_{hC}=\{\utau{}_h\in H_{h0}(\rot),\rot\utau{}_h=0\}$.
\item[{\textbf A5}] There exists an operator $\Pi_h^{\rot}:\uH^1(\Omega)\to H_{h0}(\rot)$, such that 
\begin{equation}
(\rot\Pi_h^\rot\upsi,q_h)=(\rot\upsi,q_h),\ \forall\,q_h\in L^2_{h0},\ \ \mbox{and}\ \ \ \|\uphi-\Pi_h^\rot\uphi\|_{0,\Omega}\leqslant Ch\|\uphi\|_{1,\Omega}.
\end{equation}
\end{description}

\subsubsection{Mixed element scheme for the Reissner-Mindlin plate model}
For the Reissner-Mindlin model, we consider the finite element problem: find $(\uphi{}_h^t,\uzeta{}_h^t,\omega_h^t,y_h^t,p_h^t)\in X_h:= \uH{}^1_{h0}\times H_{h0}(\rot)\times H^1_{h0}\times H^1_{hC} \times L^2_{h0}$, such that, for $(\upsi{}_h,\ueta{}_h,\mu_h,z_h,q_h)\in X_h$,
\begin{equation}\label{eq:rmmixdisgeneral}
\left\{
\begin{array}{cccccll}
(\mathcal{CE}(\uphi{}^t_h),\mathcal{E}(\upsi)) && &(\upsi{}_h,\nabla y^t_h) &+(\rot\,\upsi{}_h,p^t_h) & = \langle\uf,\upsi{}_h\rangle,
\\
&t^2(\uzeta{}^t_h,\ueta{}_h)&&+ t^2(\eta{}_h,\nabla y^t_h) &+t^2(\rot\,\ueta{}_h,p^t_h)  & =0,
\\
&& &-(\nabla\mu_h,\nabla y^t_h)&& =-\langle g,\mu_h\rangle,
\\
(\uphi{}^t_h,\nabla z_h)&+t^2(\uzeta{}^t_h,\nabla z_h) &-(\nabla\omega^t_h,\nabla z_h) && & = 0,
\\
(\rot\,\uphi{}^t_h,q_h) & +t^2(\rot\,\uzeta{}^t_h,q_h) &&& & =0.
\end{array}
\right.
\end{equation}

\begin{lemma}\label{lem:fortinhe}
Provided Assumptions {\bf A1} and {\bf A2}, given $z_h\in \mathcal{H}^1_{hC}:=\{y_h\in H^1_{hC}:(\nabla y_h,\nabla s_h)=0,\ \forall\,s_h\in H^1_{h0}\}$, there exists a $\upsi{}_h\in \uH{}^1_{h0}$, such that 
\begin{equation}
(\rot\ugamma{}_h,q_h)=0,\ \forall\,q_h\in L^2_{h0},\ \ (\ugamma{}_h,\nabla z_h)\geqslant C(\nabla z_h,\nabla z_h),\ \mbox{and}\ \ \|\ugamma{}_h\|_{1,\Omega}\leqslant C\|\nabla z_h\|_{0,\Omega}.
\end{equation}
\end{lemma}
\begin{proof}
Set $\Phi\in\mathcal{H}^2_C(\Omega)$ and $z\in \mathcal{H}^1_C(\Omega)$, such that $\Phi|_\Gamma=z|_\Gamma=z_h|_\Gamma$. Then $z_h$ is the $H^1$ projection of $z$ into $H^1_{hC}$, and $\|z-z_h\|_{1,\Omega}\leqslant Ch^{1/2+\delta_0}\|z\|_{3/2+\delta_0,\Omega}\leqslant Ch^{1/2+\delta_0}\|z\|_{1,\Omega}$. Thus $\|z_h\|_{1,\Omega}\cequiv \|z\|_{1,\Omega}\cequiv \|\Phi\|_{2,\Omega}$. Set $\ugamma{}_h=\Pi_h^{\mathsf F}\nabla\Phi$, then $(\rot\ugamma{}_h,q_h)=0$ for $q_h\in L^2_{h0}$, $\|\ugamma{}_h\|_{1,\Omega}\leqslant C\|\nabla\Phi\|_{1,\Omega}\leqslant C\|\nabla z_h\|_{0,\Omega}$, and 
\begin{multline}
(\ugamma{}_h,\nabla z_{h})=(\nabla\Phi+(\ugamma{}_h-\nabla\Phi),\nabla z+(\nabla z_{h}-\nabla z))
\\
=(\nabla\Phi,\nabla z)+((\ugamma{}_h-\nabla\Phi),\nabla z)+(\nabla\Phi,(\nabla z_{h}-\nabla z))+((\ugamma{}_h-\nabla\Phi),(\nabla z_{h}-\nabla z)).
\end{multline}
Direct calculation leads to that $(\nabla\Phi,\nabla z)=(\nabla z,\nabla z)=\|\nabla z\|_{0,\Omega}^2$; $(\nabla\Phi,(\nabla z_{h}-\nabla z))\leqslant C\|\nabla\Phi\|_{0,\Omega}h^{1/2+\delta_0}\| z\|_{1,\Omega}$; by {\bf A1}, $((\gamma{}_h-\nabla\Phi),\nabla z)\leqslant Ch\|\nabla\Phi\|_{1,\Omega}\|\nabla z\|_{0,\Omega}$; finally, $((\ugamma{}_h-\nabla\Phi),(\nabla z_{h1}-\nabla z))\leqslant Ch^{3/2+\delta_0}\|\nabla\Phi\|_{1,\Omega}\|\nabla z\|_{1,\Omega}$. By Lemma \ref{lem:findim}, summing all above leads to that $|(\ugamma{}_h,\nabla z_{h})-(\nabla z_{h},\nabla z_{h})|\leqslant |(\ugamma{}_h,\nabla z_{h})-(\nabla z,\nabla z)|+|(\nabla z,\nabla z)-(\nabla z_{h},\nabla z_{h})|\leqslant Ch^{1/2+\delta_0}\|z_{h}\|_{1,\Omega}$. The proof is completed. 
\end{proof}

\begin{lemma}\label{lem:stabmixfem}
Provided Assumptions {\bf A1}, {\bf A2} and {\bf A3}, for any $\uf\in\uH^{-1}(\Omega)$ and $g\in H^{-1}(\Omega)$, there exists a unique $(\uphi{}_h^t,\uzeta{}_h^t,\omega_h^t,y_h^t,p_h^t)\in X_h$ that solves \eqref{eq:rmmixdisgeneral}, and $$\displaystyle\|(\uphi{}_h^t,\uzeta{}_h^t,\omega_h^t,y_h^t,p_h^t)\|_{X_t}\cequiv \sup_{(\upsi{}_h,\mu_h)\in\uH{}^1_{h0}\times H^1_{h0}}\frac{\langle \uf,\upsi{}_h\rangle-\langle g,\mu_h\rangle}{\|\upsi{}_h\|_{1,\Omega}+\|\mu\|_{1,\Omega}}.$$
\end{lemma}
\begin{proof}
Again, we only have to verify Brezzi's condition for the system in \eqref{eq:rmmixdisgeneral}. The continuity and coercivity conditions are straightforward by Assumption {\bf A3}, and we are going to verify the inf-sup condition
\begin{equation}\label{eq:infsupmixdis}
\sup_{(\uphi{}_h,\uzeta{}_h,\omega_h)\in \uH{}^1_{h0}\times H_{h0}(\rot)\times H^1_{h0}\setminus\{\bf 0\}}\frac{b_t((\uphi{}_h,\uzeta{}_h,\omega_h),(z_h,q_h))}{(\|\uphi{}_h\|_{1,\Omega}+t\|\uzeta{}_h\|_{0,\Omega}+t^t\|\rot\uzeta{}_h\|_{0,\Omega}+\|\omega_h\|_{1,\Omega})}\geqslant C(\|z_h\|_{1,\Omega}+\|q_h\|_{0,\Omega}). 
\end{equation}
for any $(z_h,q_h)\in H^1_{hC}\times H^1_{h0}$. Given $(z_h,q_h)\in H^1_{hC}\times H^1_{h0}$, decompose $z_h=z_{h1}+z_{h2}$ with $z_{h1}\in\mathcal{H}^1_{hC}$ and $z_{h2}\in H^1_{h0}$. Then we choose 
\begin{itemize}
\item $\uphi{}^t_{h1}\in\undertilde{H}{}^{1}_{h0}$, such that $(\uphi{}^t_{h1},q_h)=\|q_h\|_{0,\Omega}^2$ and $\|\uphi{}^t_{h1}\|_{1,\Omega}\leqslant C\|q_h\|_{0,\Omega}$;
\item $\uphi{}^t_{h2}:=\Pi_h^{\mathsf F}\nabla \Phi$, with $\Phi\in\mathcal{H}^2_{C}$ and $\Phi|_\Gamma=z_{h1}|_\Gamma$;
\item $\uphi{}^t_h:=\uphi{}^t_{h1}+\uphi{}^t_{h2}$;
\item $\omega^t_h\in H^1_{h0}$, such that $(\nabla \omega^t_h,\nabla s_h)=(\uphi{}^t_h,\nabla s_h)-(\nabla z_{h2},\nabla s_h)$ for any $s_h\in H^1_{h0}$;
\item $\uzeta{}^t_h=\undertilde{0}$.
\end{itemize}
Then $\|\uphi{}^t_h\|_{1,\Omega}+\|\omega^t_h\|_{1,\Omega}\leqslant C(\|z_h\|_{1,\Omega}+\|q_h\|_{0,\Omega})$, and $b_t((\uphi{}^t_h,\uzeta{}^t_h,\omega^t_h),(z_h,q_h))\geqslant (q_h,q_h)+C(\nabla z_h,\nabla z_h)$ by Lemma \ref{lem:fortinhe}. The proof is completed. \end{proof}
The error estimate in energy norm follows immediately. 
\begin{theorem}\label{thm:convRM}
Provided Assumptions {\bf A1}, {\bf A2} and {\bf A3}, let $(\uphi{}^t,\uzeta{}^t,\omega^t,y^t,p^t)\in X^t$ and $(\uphi{}_h^t,\uzeta{}_h^t,\omega_h^t,y_h^t,p_h^t)\in X_h$ be the solutions of \eqref{eq:vpaug1} and \eqref{eq:rmmixdisgeneral}, respectively. With a constant $C$ uniform with respect to $t$, it holds that
\begin{multline}
\left\|\left((\uphi{}^t-\uphi{}^t_h),(\uzeta^t-\uzeta{}^t_h),(\omega^t-\omega^t_h),(y^t-y^t_h),(p^t-p^t_h)\right)\right\|_{X^t}
\\
\leqslant C\inf_{(\upsi{}_h,\ueta{}_h,\mu_h,z_h,q_h)\in X_h}  \left\|\left((\uphi{}^t-\upsi{}_h),(\uzeta^t-\ueta{}_h),(\omega^t-\mu_h),(y^t-z_h),(p^t-q_h)\right)\right\|_{X^t}.
\end{multline}
\end{theorem}

\subsubsection{A modified mixed scheme and a scheme of the primal Reissner-Mindlin plate}
Let $\Pi_h^\rot$ satisfy Assumption {\bf A5}. Now we consider a modified scheme of \eqref{eq:rmmixdisgeneral}: find $(\uphi{}^t_h,\uzeta{}^t_h,\omega^t_h,y^t_h,p^t_h)\in X_h$, such that, for $(\upsi{}_h,\ueta{}_h,\mu_h,z_h,q_h)\in X_h$,
\begin{equation}\label{eq:rmmixintdis}
\left\{
\begin{array}{cccccll}
(\mathcal{CE}(\uphi{}^t_h),\mathcal{E}(\upsi)) && &(\Pi_h^\rot\upsi{}_h,\nabla y^t_h) &+(\rot\,\upsi{}_h,p^t_h) & = \langle\uf,\upsi{}_h\rangle,
\\
&t^2(\uzeta{}^t_h,\ueta{}_h)&&+ t^2(\eta{}_h,\nabla y^t_h) &+t^2(\rot\,\ueta{}_h,p^t_h)  & =0,
\\
&& &-(\nabla\mu_h,\nabla y^t_h)&& =-\langle g,\mu_h\rangle,
\\
(\Pi_h^\rot\uphi{}^t_h,\nabla z_h)&+t^2(\uzeta{}^t_h,\nabla z_h) &-(\nabla\omega^t_h,\nabla z_h) && & = 0,
\\
(\rot\,\uphi{}^t_h,q_h) & +t^2(\rot\,\uzeta{}^t_h,q_h) &&& & =0.
\end{array}
\right.
\end{equation}

\begin{lemma}\label{lem:fortinhepirot}
Provided Assumptions {\bf A1}, {\bf A2} and {\bf A5}, given $z_h\in \mathcal{H}^1_{hC}$, there exists a $\upsi{}_h\in \uH{}^1_{h0}$, such that 
\begin{equation}
(\rot\upsi{}_h,q_h)=0,\ \forall\,q_h\in L^2_{h0},\ \ (\Pi^\rot_h\upsi{}_h,\nabla z_h)\geqslant C(\nabla z_h,\nabla z_h),\ \mbox{and}\ \ \|\upsi{}_h\|_{1,\Omega}\leqslant C\|\nabla z_h\|_{0,\Omega}.
\end{equation}
\end{lemma}
\begin{proof}
By Assumption {\bf A5} of $\Pi_h^\rot$, the proof is along the same line as that of Lemma \ref{lem:fortinhe}.
\end{proof}

\begin{lemma} 
Provided Assumptions {\bf A1}, {\bf A2}, {\bf A3} and {\bf A5}, the scheme \eqref{eq:rmmixintdis} is stable on $X_h$ (as a subspace of $X^t$).
\end{lemma}
\begin{proof}
For the continuity and coercivity, we only have to note that $\Pi_h^\rot$ is bounded from $\uH{}^1_0(\Omega)$ to $\undertilde{L}^2(\Omega)$. For the inf-sup condition, the proof is the same as that of Lemma \ref{lem:stabmixfem} by virtue of Lemma \ref{lem:fortinhepirot}. Brezzi's conditions are verified and the proof is completed.
\end{proof}

\begin{theorem}
Provided Assumptions {\bf A1}, {\bf A2}, {\bf A3} and {\bf A5}, let $(\uphi{}^t,\uzeta{}^t,\omega^t,y^t,p^t)\in X^t$ and $(\uphi{}_h^t,\uzeta{}_h^t,\omega_h^t,y_h^t,p_h^t)\in X_h$ be the solutions of \eqref{eq:vpaug1} and \eqref{eq:rmmixintdis}, respectively. Then Uniform in $t$,
\begin{multline}
\left\|\left((\uphi{}^t-\uphi{}^t_h),(\uzeta^t-\uzeta{}^t_h),(\omega^t-\omega^t_h),(y^t-y^t_h),(p^t-p^t_h)\right)\right\|_{X^t}
\\
\lesssim \inf_{(\upsi{}_h,\ueta{}_h,\mu_h,z_h,q_h)\in X_h} \left\|\left((\uphi{}^t-\upsi{}_h),(\uzeta^t-\ueta{}_h),(\omega^t-\mu_h),(y^t-z_h),(p^t-q_h)\right)\right\|_{X^t}
\\
+h(\|\uf\|_{-1,\Omega}+\|g\|_{-1,\Omega}) .
\end{multline}
\end{theorem}
\begin{proof}
By the fundamental estimation of Strang type for which we refer to, e.g., Proposition 5.5.6 of \cite{Boffi.D;Brezzi.F;Fortin.M2013}, it holds that 
\begin{multline}
\left\|\left((\uphi{}^t-\uphi{}^t_h),(\uzeta^t-\uzeta{}^t_h),(\omega^t-\omega^t_h),(y^t-y^t_h),(p^t-p^t_h)\right)\right\|_{X^t}
\\
\lesssim\inf_{(\upsi{}_h,\ueta{}_h,\mu_h,z_h,q_h)\in X_h}  \left\|\left((\uphi{}^t-\upsi{}_h),(\uzeta^t-\ueta{}_h),(\omega^t-\mu_h),(y^t-z_h),(p^t-q_h)\right)\right\|_{X^t}
\\
+\sup_{\upsi{}_h\in \uH{}^1_{h0}}\frac{(\upsi{}_h-\Pi_h^\rot\upsi{}_h,\nabla y^t)}{\|\upsi{}_h\|_{1,\Omega}}+\sup_{z_h\in H^1_{h0}}\frac{(\uphi^t-\Pi_h^\rot\uphi^t,\nabla z_h)}{\|z_h\|_{1,\Omega}}.
\end{multline}
By Assumption {\bf A5}, 
$$
\sup_{\upsi{}_h\in \undertilde{H}{}_{h0}^1}\frac{(\upsi{}_h-\Pi_h^\rot\upsi{}_h,\nabla y^t)}{\|\upsi{}_h\|_{1,\Omega}}\leqslant Ch\|\nabla y^t\|_{0,\Omega},\ \mbox{and}\ \sup_{z_h\in H^1_{h0}}\frac{(\uphi^t-\Pi_h^\rot\uphi^t,\nabla z_h)}{\|z_h\|_{1,\Omega}}\leqslant Ch\|\uphi^t\|_{1,\Omega}.
$$
The proof is completed by noting Theorem \ref{thm:wprmmix}. 
\end{proof}
\paragraph{\bf A scheme of the primal Reissner-Mindlin plate} We begin with the observation below. 
\begin{lemma}
Provided Assumption {\bf A4}, let $(\uphi{}_h^t,\uzeta{}_h^t,\omega_h^t,y_h^t,p_h^t)\in X_h$ be the solution of \eqref{eq:rmmixintdis}, then $\uzeta{}^t_h=t^{-2}(\nabla\omega^t_h-\Pi^\rot_h\uphi{}^t_h)$.
\end{lemma}
A primal finite element scheme can be given as: find $(\uphi{}^t_h,\omega{}^t_h)\in \undertilde{H}{}^{1}_{h0}\times H^1_{h0}$, such that 
\begin{equation}\label{eq:DLform}
(\mathcal{CE}(\uphi{}^t_h),\mathcal{E}(\upsi{}_h))+t^{-2}(\Pi_h^\rot\uphi{}^t_h-\omega^t_h,\Pi_h^\rot\upsi{}_h-\mu_h)=\langle\uf,\upsi{}_h\rangle+\langle g,\mu_h\rangle,\ \forall\,(\upsi{}_h,\mu{}_h)\in \undertilde{H}{}^{1}_{h0}\times H^1_{h0}.
\end{equation}

Provided Assumptions {\bf A1} through {\bf A5}, the primal scheme \eqref{eq:DLform} is equivalent to the scheme \eqref{eq:rmmixintdis}, in the sense below: 
\begin{enumerate}
\item let $(\uphi{}_h^t,\uzeta{}_h^t,\omega_h^t,y_h^t,p_h^t)\in X_h$ be the solution of \eqref{eq:rmmixintdis}, then $(\uphi{}_h^t,\omega_h^t)$ solves \eqref{eq:DLform};
\item as evidently, the solution of \eqref{eq:DLform}, if existent, is unique, any solution of \eqref{eq:DLform} is part of a solution of \eqref{eq:rmmixintdis}. 
\end{enumerate}

\subsubsection{Mixed element schemes of the Kirchhoff plate}

Associated to \eqref{eq:rmmixdisgeneral}, a finite element scheme for the Kirchhoff plate model is: find $(\uphi{}_h,\omega_h,y_h,p_h)\in Y_h:=\uH{}^1_{h0}\times H^1_{h0}\times H^1_{hC} \times L^2_{h0}$, such that, for $(\upsi{}_h,\mu_h,z_h,q_h)\in Y_h$,
\begin{equation}\label{eq:kmixdis}
\left\{
\begin{array}{ccccll}
(\mathcal{CE}(\uphi{}_h),\mathcal{E}(\upsi)) & &(\upsi{}_h,\nabla y_h) &+(\rot\,\upsi{}_h,p_h) & = \langle\uf,\upsi{}_h\rangle,
\\
&&-(\nabla\mu_h,\nabla y_h)&& =-\langle g,\mu_h\rangle,
\\
(\uphi{}_h,\nabla z_h)&-(\nabla\omega_h,\nabla z_h) && & = 0,
\\
(\rot\,\uphi{}_h,q_h) &  && & =0.
\end{array}
\right.
\end{equation}

\begin{lemma}\label{lem:convK}
Provided Assumptions {\bf A1} and {\bf A2}, the scheme \eqref{eq:kmixdis} is stable on $Y_h$. Let $(\uphi,\omega,y,p)\in Y$ and $(\uphi{}_h,\omega_h,y_h,p_h)\in Y_h$ be the solutions of  \eqref{eq:mixkp} and \eqref{eq:kmixdis}, respectively. There is a constant $C$, uniform with respect to $t$ and $h$, such that
\begin{multline}
\|\uphi-\uphi{}_h\|_{1,\Omega}+\|\omega-\omega_h\|_{1,\Omega}+\| y-y_h\|_{1,\Omega}+\|p-p_h\|_{0,\Omega}
\\
\leqslant C\inf_{(\upsi{}_h,\mu_h,z_h,q_h)\in Y_h} (\|\uphi-\upsi{}_h\|_{1,\Omega}+\|\omega-\mu_h\|_{1,\Omega}+\| y-z_h\|_{1,\Omega}+\|p-q_h\|_{0,\Omega}).
\end{multline}
\end{lemma}

A variant is presented as: find $(\uphi{}_h,\omega_h,y_h,p_h)\in Y_h$, such that, for $(\upsi{}_h,\mu_h,z_h,q_h)\in Y_h$,
\begin{equation}\label{eq:kmixdisint}
\left\{
\begin{array}{ccccll}
(\mathcal{CE}(\uphi{}_h),\mathcal{E}(\upsi)) & &(\Pi_h^\rot\upsi{}_h,\nabla y_h) &+(\rot\,\upsi{}_h,p_h) & = \langle\uf,\upsi{}_h\rangle,
\\
&&-(\nabla\mu_h,\nabla y_h)&& =-\langle g,\mu_h\rangle,
\\
(\Pi_h^\rot\uphi{}_h,\nabla z_h)&-(\nabla\omega_h,\nabla z_h) && & = 0,
\\
(\rot\,\uphi{}_h,q_h) &  && & =0.
\end{array}
\right.
\end{equation}

\begin{lemma}\label{lem:convKint}
Provided Assumptions {\bf A1}, {\bf A2} and {\bf A5}, the scheme \eqref{eq:kmixdisint} is stable on $Y_h$. Let $(\uphi,\omega,y,p)\in Y$ and $(\uphi{}_h,\omega_h,y_h,p_h)\in Y_h$ be the solutions of  \eqref{eq:mixkp} and \eqref{eq:kmixdisint}, respectively. There is a constant $C$, uniform with respect to $t$ and $h$, such that
\begin{multline}
\|\uphi-\uphi{}_h\|_{1,\Omega}+\|\omega-\omega_h\|_{1,\Omega}+\| y-y_h\|_{1,\Omega}+\|p-p_h\|_{0,\Omega} \leqslant C(h(\|\uf\|_{-1,\Omega}+\|g\|_{-1,\Omega})
\\
+\inf_{(\upsi{}_h,\mu_h,z_h,q_h)\in Y_h} (\|\uphi-\upsi{}_h\|_{1,\Omega}+\|\omega-\mu_h\|_{1,\Omega}+\| y-z_h\|_{1,\Omega}+\|p-q_h\|_{0,\Omega})).
\end{multline}
\end{lemma}

\subsection{An example of finite element space quintuple}

For $K$ a triangle, we use $P_k(K)$ for the set of polynomials on $K$ of degrees not higher than $k$. Denote by $a_i$ and $E_i$ vertices and opposite edges of $K$, $i=1,2,3$.  The barycentre coordinates are denoted as usual by $\lambda_i$, $i=1,2,3$. Besides, define shape function spaces  $P^e(K):={\rm span}\{\lambda_i\lambda_j,1\leqslant i\neq j\leqslant 3\}$ and $\mathbb{E}(K):=\{\uu+v\ux^\perp:\uu\in \mathbb{R}^2,\ v\in\mathbb{R}\}$. Let $\mathcal{G}_h$ be a shape-regular triangular subdivision of $\Omega$, such that $\overline\Omega=\cup_{K\in\mathcal{G}_h}\overline K$. Denote by $\mathcal{E}_h$, $\mathcal{E}_h^i$, $\mathcal{X}_h$ and $\mathcal{X}_h^i$ the set of edges, interior edges, vertices and interior vertices, respectively. For any edge $e\in\mathcal{E}_h$, denote by $\undertilde{t}{}_e$ the unit tangential vector along $e$. Define finite element spaces as
\begin{itemize}
%
\item $\mathsf{L}_h:=\{w\in H^1(\Omega):w|_K\in P_1(K),\ \forall\,K\in \mathcal{G}_h\}$;

\item $\mathsf{L}_{hC}:=\mathsf{L}_h\cap H^1_C(\Omega)$; $\mathsf{L}_{h0}:=\mathsf{L}_h\cap H^1_0(\Omega)$; $\undertilde{\mathsf{L}}{}_{h}:=(\mathsf{L}_{h})^2$; $\undertilde{\mathsf{L}}{}_{h0}:=(\mathsf{L}_{h0})^2$;


\item $\mathsf{L}_h^{e}:=\{w\in H^1(\Omega):w|_K\in P^{e}(K),\ \forall\,K\in \mathcal{G}_h\}$, and $\mathsf{L}_{h0}^{e}:=\mathsf{L}_h^{e}\cap H^1_0(\Omega)$;



\item $b_e\in \mathsf{L}^{e}_h$ such that $b_{e}=0$ on $e'\in\mathcal{E}_h\setminus \{e\}$; 

\item $\undertilde{\mathsf{L}}{}_h^e:={\rm span}\{b_e\undertilde{t}{}_e\}_{e\in \mathcal{E}_h}$, and $\undertilde{\mathsf{L}}{}_{h0}^e=\undertilde{\mathsf{L}}{}_h^e\cap \undertilde{H}{}^1_0(\Omega)$;


\item $\undertilde{\mathsf L}{}_{h0}^{+e}:=\undertilde{\mathsf L}{}_{h0}+\undertilde{\mathsf L}{}_{h0}^e$; 


\item $\mathbb{R}_h:=\{\uw\in H(\rot,\Omega):\uw|_K\in \mathbb{E}(K),\ \forall\,K\in \mathcal{G}_h\}$, and $\mathbb{R}_{h0}:=\mathbb{R}_{h}\cap H_0(\rot,\Omega)$.

\item $\mathcal{L}^0_{h}:=$ space of piecewise constant, and $\mathcal{L}^0_{h0}:=\mathcal{L}^0_{h}\cap L^2_0(\Omega)$.
\end{itemize}

The space $\undertilde{\mathsf L}{}^{+e}_{h0}$ is the rotation of the Bernardi-Raugel element space \cite{Bernardi.C;Raugel.G1985}, and $\mathbb{R}_{h0}$ is the rotated Raviart-Thomas element space \cite{Raviart.P;Thomas.J1977} of lowest order. 

\begin{lemma}
The exact sequence holds that 
\begin{equation}
\begin{array}{ccccccccc}
0 & \xrightarrow{\rm Id} & \mathsf{L}_{hC} & \xrightarrow{\bs{\mrm{grad}}} & \mathbb{R}_{h0} & \xrightarrow{\mrm{rot}} & \mathcal{L}_{h0}^0  & \xrightarrow{\int_\Omega\cdot}& 0.
\end{array}
\end{equation}
\end{lemma}
\begin{proof}
It is evident that $\rot\mathbb{R}_{h0}\subset \mathcal{L}_{h0}^0$. Given $q_h\in \mathcal{L}_{h0}^0$, we have for $\utau{}_h\in\mathbb{R}_{h0}$ that $\displaystyle\int_\Omega\rot\utau{}_hq_h=\sum_{e\in\mathcal{E}_h^i}\int_e\utau{}_h\cdot\ut{}_e\llbracket q_h\rrbracket|_e.$ Thus, with $\utau{}_h$ such that $\int_e\utau{}_h\cdot\ut{}_e=\llbracket q_h\rrbracket|_e$, we have $\int_\Omega\rot\utau{}_hq_h>0$ for $q_h\not\equiv 0$. This implies $\rot\mathbb{R}_{h0} = \mathcal{L}_{h0}^0$. By Euler's formula for multiply-connected domains, $\dim(\mathbb{R}_{h0})=\dim(\mathsf{L}_{hC})+\dim(\mathcal{L}_{h0}^0)$.  Meanwhile, $\nabla$ maps $\mathsf{L}_{hC}$ into $\mathbb{R}_{h0}$ (specifically the kernel of $\rot$) injectively. It follows that $\nabla\mathsf{L}_{hC}$ is the kernel space of $\rot$ in $\mathbb{R}_{h0}$. The proof is completed. 
\end{proof}

Denote by $\Pi_h^C$ the Clement operator from $\uH{}^1_{0}(\Omega)$ to $\undertilde{\mathsf{L}}{}_{h0}$, and define $\Pi_h^{+e}$ the interpolation from $\uH{}^1_{0}(\Omega)$ to $\undertilde{\mathsf{L}}{}_{h0}^{+e}$ by
\begin{equation}
(\Pi_h^{+e}\uphi)(\ux)=(\Pi_h^C\uphi)(\ux),\ \forall\,\ux\in \mathcal{X}_h,\  \mbox{and}\ \int_e (\Pi_h^{+e}\uphi)\cdot\ut{}_e=\int_e \uphi\cdot\ut{}_e,\ \forall\,e\in\mathcal{E}_h.
\end{equation} 
Then $(\rot \Pi_h^{+e}\uphi,q_h)=(\rot\uphi,q_h)$ for any $q_h\in\mathcal{L}_{h0}^0$, and $|\uphi-\Pi_h^{+e}\uphi|_{k,\Omega}\leqslant Ch^{1-k}|\uphi|_{1,\Omega}$, $k=0,1$.

Let $\Pi_h^\mathbb{R}$ be the nodal interpolator of $\mathbb{R}_{h0}$ defined by, with $\utau$ regular enough,
$$
\int_e\Pi_h^\mathbb{R}\utau\cdot\ut{}_e=\int_e\utau\cdot\utau{}_e.
$$
By standard error estimate (c.f., e.g., Proposition 2.5.4 of \cite{Boffi.D;Brezzi.F;Fortin.M2013}) $\|\utau-\Pi_h^\mathbb{R}\utau\|_{0,\Omega}\leqslant Ch\|\utau\|_{1,\Omega}$.

Now set
$$
\begin{array}{cccc}
\uH{}^1_{h0}=\undertilde{\mathsf L}{}^{+e}_{h0};& H^1_h=\mathsf{L}_h;& H^1_{hC}=\mathsf{L}_{hC};& H^1_{h0}=\mathsf{L}_{h0};
\\
 H_{h0}(\rot)=\mathbb{R}_{h0}; &L^2_{h0}=\mathcal{L}^0_{h0}; & \Pi_h^{\mathsf F}=\Pi_h^{+e}; & \Pi_h^\rot=\Pi_h^\mathbb{R}.
\end{array}
$$
Then Assumptions {\bf A1} through {\bf A5} are all satisfied. 

\begin{remark}
With the finite element spaces defined above, the scheme \eqref{eq:DLform} coincides with the one proposed by Dur\'an-Liberman \cite{Duran.R;Liberman.E1992}. 
\end{remark}

\section{Concluding remarks}
\label{sec:conc}

In this paper, the Reissner-Mindlin thick plate model and the Kirchhoff thin plate model are studied on multiply-connected polygonal domains. Equivalent mixed formulations of the plate models are presented, and uniform stability analysis of the mixed systems is given. A framework is designed for discretizing the mixed formulations, with uniform stability constructed analogously under some conditions. The error estimation in energy norm is constructed with respect to the assumption of the regularity of the solution. An example is given to verify the framework. For schemes that fall in the framework, an optimal diagonal preconditioner can be constructed by virtue of \cite{Ruesten.T;Winther.R1992,Hiptmair.R;Xu.J2007}.

The mixed system suggested for the Reissner-Mindlin plate involve five variables; restricted on simply-connected domain, the quintuple of spaces is the same as the quintuple used in \cite{Brezzi.F;Fortin.M;Stenberg.R1991}. Due to the harmonic function existing in the Helmholtz decomposition of $H^{-1}(\dv,\Omega)$, big systems constructed on the quintuple on multiply-connected domain can not generally be decomposed to small subsystems. This way, the system constructed on the space quintuple has generally to be studied as an entire one, and an approach different from that of \cite{Brezzi.F;Fortin.M1986,Brezzi.F;Fortin.M;Stenberg.R1991,Arnold.D;Falk.R1989} is utilised. 

A stable regular decomposition $H^1_0(\rot,\Omega)=\nabla H^1_0(\Omega)+\uH{}^1_0(\Omega)$ on multiply-connected polygon is constructed, and its discretised analogue can indeed be constructed. This can be viewed as another main ingredient of the paper. Other regular decompositions on regular domains in two and three dimensional can be investigated the similar way on domains not that regular with applications in designing discretization schemes and optimal preconditioners. This will be discussed in future.  We refer to \cite{Zhang.S2016fw} for some related discussion.

As an example of the framework, the Dur\'an-Liberman's scheme \cite{Duran.R;Liberman.E1992} for Reissner-Mindlin plate originally developed on simply-connected domains is extended to multiply-connected domains. It will be naturally expected that some other existing schemes for simply-connected domain could be extended to multiply-connected domains by the aid of the new approach. Some conforming or nonconforming schemes in MITC type, like ones in, e.g., \cite{Ming.P;Shi.Z2001,Hu.J;Ming.P;Shi.Z2003,Bathe.K;Dvorkin.E1985}, can be studied in future.

In this paper, the error estimation in energy norm is given with respect to the regularity of the solution. For Kirchhoff model, a regularity of the system may be obtained by the aid of regularity theory of Poisson and Stokes systems on corner domain. Meanwhile, for Reissner-Mindlin plate, concise regularity analysis of the solution will be in need for a robust error estimation, especially estimation in low-order norms. The regularity analysis is well interacted with the asymptotic analysis between thin and moderately thick plates, but due to the limited regularity of the domain considered, the techniques for regularity analysis and asymptotic analysis of the solution in \cite{Brezzi.F;Fortin.M1986,Arnold.D;Falk.R1997,Pitkaranta.J1988,Arnold.D;Falk.R1989} can hardly be directly repeated. These will have to be discussed in future in an integrated way, and perhaps firstly for non convex simply-connected polygons.

\section*{Acknowledgment}
The author is supported partially by the National Natural Science Foundation of China with Grant No. 11471026 and National Centre for Mathematics and Interdisciplinary Sciences, Chinese Academy of Sciences.

\end{document}